\documentclass[a4paper,11pt,reqno]{amsart}
\usepackage{amssymb}
\usepackage{fancyhdr}
\usepackage[margin=2.9cm]{geometry}
\usepackage{hyperref}
\usepackage{enumitem}

\hypersetup{
    colorlinks=true,
    linkcolor=blue,
    citecolor=blue,      
    urlcolor=black,
}

\newtheorem{Th}{Theorem}
\newtheorem{Lemma}[Th]{Lemma}
\newtheorem{Cor}[Th]{Corollary}

\allowdisplaybreaks[1]

\begin{document}

\pagestyle{fancy}
\lhead{}
\chead{}
\rhead{\thepage}
\cfoot{}
\lfoot{}
\rfoot{}
\renewcommand{\headrule}{}

\title{On a specific family of orthogonal polynomials of Bernstein-Szeg\"o type} 

\author{Martin Nicholson} 

\begin{abstract}  
We study a class of weight functions on $[-1,1]$, which are special cases of the general weights studied by Bernstein and Szeg\"o, as well as their extentions to the interval $[-a,1]$ for a continuous parameter $a>0$. These weights are parametrized by two positive integers. As these integers tend to infinity, these weights approximate certain weight functions on $\mathbb{R}$ considered in the earlier literature in connection with orthogonal polynomials related to elliptic functions. It turns out that an orthogonal polynomial of certain degree corresponding to these weights has a particularly simple form with known roots. This fact allows us to find explicit quadrature formulas for these weights and construct measures on $\mathbb{R}$ with identical moments. We also find finite analogs of some improper integrals first studied by Glaisher and Ramanujan, and show that some of the  functions used in this work are in fact generating functions of certain finite trigonometric power sums.
\end{abstract}

\clearpage\maketitle
\thispagestyle{empty}
\vspace{-30pt}

\section{Introduction}

Entry 4.123.6 in Gradshteyn and Ryzhik's {\it Table of Integrals, Series, and Products} \cite{gr} reads
\begin{equation}\label{gr}
    \int\limits_0^\infty \frac{\sin (ax)\sinh (bx)}{\cos (2 ax)+\cosh \left(2bx\right)}\,x^{p-1}dx=\frac{\Gamma(p)}{(a^2+b^2)^{p/2}}\,\sin\left(p\tan^{-1}\frac{a}{b}\right)\sum_{k=0}^\infty\frac{(-1)^k}{(2k+1)^p},\quad p>0.
\end{equation}
The limit $p\to+0$ of \eqref{gr} is
\begin{equation}\label{arctan1}
    \int\limits_0^\infty \frac{\sin (x)\sinh (x/a)}{\cos (2 x)+\cosh \left(2x/a\right)}\frac{dx}{x}=\frac{\tan^{-1} a}{2}.
\end{equation}
Additionally, when $a=b$ and $p/4\in\mathbb{N}$ one finds from \eqref{gr}
\begin{equation}\label{z}
    \int\limits_0^\infty \frac{\sin (x)\sinh (x)}{\cos (2 x)+\cosh \left(2x\right)}\,x^{4k-1}dx=0,\quad k\in\mathbb{N}.
\end{equation}
Integrals of the type \eqref{arctan1} and \eqref{z} were first studied by Glaisher \cite{glaisher}. More recently, they were studied in connection with integrals of the Dedekind eta-function \cite{glasser},\cite{patkowski},\cite{coffey}.

In \cite{nicholson}, we have generalized \eqref{arctan1} as
\begin{equation}\label{form1}
\int\limits_0^{1}\frac{\sin \bigl(n \sin^{-1}{t}\bigr)\sinh \bigl(n \sinh^{-1}({t}/a)\bigr)}{\cos \bigl( 2 n \sin^{-1}{t}\bigr)+\cosh \bigl(2 n \sinh^{-1}({t}/a)\bigr)}\frac{dt}{t \sqrt{(1-t^2)(1+t^2/a^2)}}=\frac{\tan^{-1} a}{2},
\end{equation}
where $n$ is a positive odd integer, and we have also shown that
\begin{equation}\label{coscosheven}
    \int\limits_{0}^{1}\frac{\cos (n \sin^{-1}{t})\cosh (n \sinh^{-1}{t})}{\cos (2 n \sin^{-1}{t})+\cosh (2 n \sinh^{-1}{t})}\,\frac{tdt}{\sqrt{1-t^4}}=0,
\end{equation}
where $n$ is a positive even integer. Our proof was based on explicit calculations using the fact that the roots of $\cos \bigl( 2 n \sin^{-1}t\bigr)+\cosh \bigl(2 n \sinh^{-1}(t/a)\bigr)$ (which is a polynomial in $t$) can be determined in closed form (see also Sections \ref{rq} and \ref{theta1} of the present paper for similar calculations). Two alternative proofs of \eqref{form1} were given in \cite{teruo} and \cite{po1ynomial}. Motivation for considering such integrals came from the mapping $\alpha_z=2n\sinh^{-1}\sin\frac{\pi z}{2n}$ encountered in the theory of Dirichlet problem on finite nets \cite{pw}, as discussed in Section 5 of \cite{nicholson}. Note that $\alpha_z\sim \pi z$, when $n\to\infty$.

In this paper, we will be concerned with integration formulas similar to the following:
\begin{Th}\label{th1.0} Let $n$ and $m$ be positive odd integers. Then
    \begin{equation}\label{f1}
    \int\limits_{-1}^{1}\frac{\sin \bigl(n \sin^{-1}\!\sqrt{t}\bigr)\sinh \bigl(m \sinh^{-1}\!\sqrt{t}\bigr)}{\cos \bigl(2 n \sin^{-1}\!\sqrt{t}\bigr)+\cosh \bigl(2 m \sinh^{-1}\!\sqrt{t}\bigr)}\frac{t^jdt}{\sqrt{1-t^2}}=\begin{cases}
        \pi/2,\quad &j=-1,\\
        0,\quad &j=0,1,\ldots,\frac{m+n-2}{2},
    \end{cases}
\end{equation}
\begin{equation}\label{f2}
    \nonumber\int\limits_{-1}^{1}\frac{\cos \bigl(n \sin^{-1}\!\sqrt{t}\bigr)\cosh \bigl(m \sinh^{-1}\!\sqrt{t}\bigr)}{\cos \bigl(2 n \sin^{-1}\!\sqrt{t}\bigr)+\cosh \bigl(2 m \sinh^{-1}\!\sqrt{t}\bigr)}\,t^jdt=0,\quad j=0,1,\ldots,\tfrac{m+n-4}{2}.
\end{equation}
\end{Th}
For generic integers $n$ and $m$, the roots of the polynomial $\cos \bigl( 2 n \sin^{-1}t\bigr)+\cosh \bigl(2 m \sinh^{-1}t\bigr)$ are not known in closed form. Therefore other methods are necessary to study these more general integrals. 

Our proof of Theorem \ref{th1.0} is based on the theory of orthogonal polynomials in its elementary form. In particular, we will use Bernstein-Szeg\"o orthogonal polynomials, a brief overview of which is given in Section \ref{sb}. After studying the aforementioned function $\cos \bigl( 2 n \sin^{-1}t\bigr)+\cosh \bigl(2 m \sinh^{-1}t\bigr)$ in Section \ref{function}, and proving Theorem \ref{th1.0} in Section \ref{pth1}, we turn to more general situation and consider modifications of Theorem \ref{th1.0} by a continuous parameter $a>0$ in Section \ref{a}. The case where both parameters $n$ and $m$ are even is studied in Section \ref{even}. More complicated weight functions are studied in Section \ref{d}.

In a series of papers \cite{berg},\cite{ismail}, Berg, Valent, and Ismail have considered orthogonal polynomials on $\mathbb{R}$ related to elliptic functions. Weight function for these orthogonal polynomails is equivalent to
\begin{equation}\label{iv}
    \frac{1}{\cos \bigl(2 \sqrt{x}\bigr)+\cosh \bigl(2\sqrt{x/a}\bigr)}
\end{equation}
after rescaling of the variable. Recently, there has been a flurry of activity in studying different aspects of the integrals with the weight function \eqref{iv}, e.g. \cite{berndt},\cite{kuznetsov},\cite{xu},\cite{bradshaw},\cite{pan},\cite{rui},\cite{zhou}.
The weight functions
\begin{equation}\label{wfmn}
    \frac{1}{\cos \bigl(2 n \sin^{-1}\!\sqrt{t}\bigr)+\cosh \bigl(2 m \sinh^{-1}\!\sqrt{t}\bigr)}\frac{1}{\sqrt{1-t^2}},\qquad t\in[-1,1],
\end{equation}
parametrized by two positive integers $n$ and $m$, approximate the weight functions \eqref{iv} in the limit ${n,m\to\infty}$. For example, \eqref{f1} is a finite analog of the improper integral
\[
\int\limits_{\mathbb{R}} \frac{\sin\bigl(\sqrt{x}\bigr)\sinh \bigl(\sqrt{x/a}\bigr)}{\cos \bigl(2 \sqrt{x}\bigr)+\cosh \bigl(2\sqrt{x/a}\bigr)}\,x^{j}dx=\begin{cases}
    \pi/2,\quad &j=-1,\\
    0,\quad &j=0,1,2,\ldots
\end{cases}
\]
As can be seen from \eqref{f1}, $\sin (n \sin^{-1}\!\sqrt{t})\sinh (m \sinh^{-1}\!\sqrt{t})$ is a degree $\frac{m+n}{2}$ orthogonal polynomial corresponding to the weight function \eqref{wfmn}. Moreover, the roots of this orthogonal polynomial are known. This will allow us to find explicit Gauss quadrature formulas for the weight function \eqref{wfmn} in Section \ref{gq}. By applying these quadrature formulas to particular polynomials in Section \ref{section_generating_functions}, we derive a finite analog of the generating function formula due to Kuznetsov \cite{kuznetsov}. 

In Section \ref{contructing_measures}, we find extentions of certain formulas from \cite{lubinsky2} and then use these results to construct measures on $\mathbb{R}$ that have the same first $n+m-1$ moments as \eqref{wfmn}. Then, considering the limit of large $n,m$ for these measures, we establish connection of our work with \cite{kuznetsov2}.

Consider the integral for a positive odd integer $k$
\begin{equation}\label{353}
    \int\limits_0^\infty \frac{\sin (kx)}{\cos (x)+\cosh \left(x\right)}\,\frac{dx}{x}=\frac{\pi}{4},
\end{equation}
which was submitted by Ramanujan to the Journal of the Indian Mathematical Society as a problem number $353$ \cite{ramanujan}. More information on the history of \eqref{353} can be found in \cite{berndt}. 
In Section \ref{rq}, we will prove a finite analog of \eqref{353} that contains an additional integer parameter. When this integer parameter goes to infinity, one recovers \eqref{353} in the limit. 

In Section \ref{theta1}, a calculation similar to that of Section \ref{rq} is used to evaluate a certain integral in terms of a finite trigonometric sum. This turns out to be a finite analog of another integral due to Glaisher \cite{glaisher} related to theta series. We also show that this analog is equivalent to generating function formula for a certain finite trigonometric sum.

The indeterminate moment problem for orthogonal polynomails associated with the weight functions 
\[
 \frac{x}{\cosh \bigl(2\sqrt{x/a}\bigr)-\cos \bigl(2 \sqrt{x}\bigr)}
\]
has been studied in \cite{ivy}. We briefly consider finite analogs of such weight functions in Section \ref{appendix}.

\section{General Bernstein-Szego polynomials}\label{sb}

In this section, we closely follow the book \cite{szego}. The trigonometric polynomial in $\theta$ of degree $k$ is
\[
g(\theta)=a_0+\sum_{j=1}^k\left\{a_j\cos(j\theta)+b_j\sin(j\theta)\right\}.
\]
\begin{Th}[\cite{szego}, Theorem 1.2.2]\label{th1.2.2}
    Let $g(\theta)$ be a trigonometric polynomial with real coefficients which is nonnegative for all real values of $\theta$ and $g(\theta)\not\equiv 0$. Then a representation $g(\theta)=|h(e^{i\theta})|^2$ exists such that $h(z)$ is a polynomial of the same degree as $g(\theta)$, with $h(z)\neq 0$ in $|z|<1$, and $h(0)>0$. This polynomial is uniquely determined. If $g(\theta)$ is a cosine polynomial, $h(z)$ is a polynomial with real coefficients.
\end{Th}
Let $\rho(t)$ be a polynomial of precise degree $l$ and positive in $[-1,1]$. Then, orthonormal polynomials $p_k(t)$ (of degree $k$), which are associated with weight functions
\[
w(t)=\frac{1}{\rho(t)\sqrt{1-t^2}},\qquad t\in[-1,1],
\]
can be calculated explicitly provided $l<2k$. Namely, let $\rho(\cos\theta)=|h(e^{i\theta})|^2$ be the normalized representation of $\rho(\cos\theta)$ in the sense of Theorem \ref{th1.2.2}. Then, writing $h(e^{i\theta})=c(\theta)+is(\theta)$, $c(\theta)$ and $s(\theta)$ real, we have
\begin{equation}\label{pk}
    p_k(t)=\sqrt{\frac{2}{\pi}}\operatorname{Re}\bigl\{e^{ik\theta}\,\overline{h(e^{i\theta})}\bigr\}
    =\sqrt{\frac{2}{\pi}}\,\bigl\{c(\theta)\cos k\theta+s(\theta)\sin k\theta\bigr\}.
\end{equation}
These formulas must be modified for $l=2k$ by multiplying the right-hand member of \eqref{pk} by a certain constant factor. In this paper, we will only consider $l<2k$.

The other two weight functions for which the corresponding orthonormal polynomials can be calculated explicitly are:
\begin{alignat}{3}
w(t)&=\frac{1}{\rho(t)}\sqrt{1-t^2},\qquad &&p_k(t)=\sqrt{\frac{2}{\pi}}\,\frac{\operatorname{Im}\bigl\{e^{i(k+1)\theta}\,\overline{h(e^{i\theta})}\bigr\}}{\sin\theta},\qquad &&l<2k+2;\label{pk1}\\
w(t)&=\frac{1}{\rho(t)}\sqrt{\frac{1-t}{1+t}},\qquad &&p_k(t)=\sqrt{\frac{2}{\pi}}\,\frac{\operatorname{Im}\bigl\{e^{i(k+1/2)\theta}\,\overline{h(e^{i\theta})}\bigr\}}{\sin(\theta/2)},\qquad &&l<2k+1.\label{pk2}
\end{alignat}

\section{The function \texorpdfstring{$\cos (2 n \sin^{-1}\!\sqrt{t})+\cosh (2 m \sinh^{-1}\!\sqrt{t})$}{rho(t)}}\label{function}

The integral $I(n,m,j)$ on the left-hand side of \eqref{f1} satisfies
\[
I(n,m,j)=(-1)^{j-1}\,I(m,n,j).
\]
This symmetry means that it is enough to consider $m\ge n$. The expression
\begin{align*}
    \rho(t)&=\cos \bigl(2 n \sin^{-1}\!\sqrt{t}\bigr)+\cosh \bigl(2 m \sinh^{-1}\!\sqrt{t}\bigr)\\
    &=T_n(1-2t)+T_m(1+2t),
\end{align*}
(where $T_n(x)=\cos(n\cos^{-1}x)$ are Chebyshev polynomials of the first kind) is a polynomial in $t$ of degree
\[
\operatorname{deg}\rho=\begin{cases}
    m, &m> n,\\
    2\lfloor\frac{m}{2}\rfloor, &m=n.
\end{cases}
\]
One can write for $t\in[-1,1]$
\[
\rho(t)=\bigl|\sqrt{2}\,\cos \bigl( n \sin^{-1}\sqrt{t}-im \sinh^{-1}\sqrt{t}\bigr)\bigr|^2.
\]
Using logarithmic form of the functions $\sin^{-1}$, $\sinh^{-1}$, we obtain the representation
\[
\sqrt{2}\,\cos \bigl( n \sin^{-1}\sqrt{\cos \theta}-im \sinh^{-1}\sqrt{\cos \theta}\bigr)=i^ne^{-i (n+m)\theta/2}\,h\bigl(e^{i\theta}\bigr),
\]
\begin{align}\label{h}
    \nonumber h(z)=\frac{1}{\sqrt{2}}\bigg\{&\left(\sqrt{z^2+1}+1\right)^{\frac{m+n}{2}}\left(\sqrt{z^2+1}+z\right)^{\frac{m-n}{2}}\\
    &+(-1)^n\!\left(\sqrt{z^2+1}-1\right)^{\frac{m+n}{2}} \left(\sqrt{z^2+1}-z\right)^{\frac{m-n}{2}}\bigg\}.
\end{align}
By construction, $h(z)$ is a polynomial in $z$. An alternative form for $h(z)$ is
\[
h(z)=2^{\frac{1-m-n}{2}}\Bigg\{\sum_{j=0}^n\sum_{s=0}^{j/2}\binom{m+n}{m+j}\binom{j}{2s}(-z)^{j-2s}\bigl(1+z^2\bigr)^{s}+\sum_{j=1}^m\sum_{s=0}^{j/2}\binom{m+n}{n+j}\binom{j}{2s}z^{j-2s}\bigl(1+z^2\bigr)^{s}\Bigg\}.
\]
\begin{Lemma}\label{lemma}
\begin{enumerate}[label=\textnormal{(\roman*)}, wide=0pt]
    \item $\operatorname{deg}h=\operatorname{deg}\rho$.
    \item $h(0)>0$.
    \item $h(z)\neq 0$ in $|z|<1$.
\end{enumerate}
\end{Lemma}
\begin{proof} Parts (i) and (ii) are obvious. Part (iii) can be proved by Rouch\'e's theorem as follows. Let
\[
f(z)=(-1)^n\,z^{-m-n}\left(1+\sqrt{z^2+1}\right)^{m+n} \left(z+\sqrt{z^2+1}\right)^{m-n}.
\]
Unlike $h(z)$, $f(z)$ is a multivalued function. We choose brunch cuts on the rays $[i,+i\infty)$ and $[-i,-i\infty)$.
The roots of $h(z)$ coincide with the roots of the equation $f(z)-1=0$. Consider the contour $C$ composed of four arcs: two arcs of unit radius centered at the origin, and two arcs of small radius $\varepsilon>0$ centered at $\pm i$. We will show that $|f(z)|>1$ on $C$. Since $f(z)$ does not have zeroes inside the unit circle, according to Rouch\'e's theorem it will follow that $f(z)-1$ does not vanish inside the unit circle.

One can easily show that $|f(z)|>1$ when $|z|=1$, with the exception of the two points $\pm i$. The arc around $+i$ can be parametrized as
\[
z=i+\varepsilon e^{-2i\varphi},\quad \varphi\in (0,\pi/2).
\]
Due to
\[
\left|f\bigl(i+\varepsilon e^{-2i\varphi}\bigr)\right|=1+2\sqrt{\varepsilon}\left( m\cos{\varphi}+n\sin{\varphi}\right)+O(\varepsilon) ,\quad \varphi\in (0,\pi/2),
\]
and since $m\cos{\varphi}+n\sin{\varphi}$ is strictly positive for $\varphi\in (0,\pi/2)$ with positive $m$, $n$, we deduce that $|f(z)|>1$ on the arc around $+i$. The arc around $-i$ is dealt with in the same manner. \end{proof}
Hence, $\rho(\cos\theta)=|h(e^{i\theta})|^2$ is the normalized representation of $\rho(\cos\theta)$ in the sense of Theorem \ref{th1.2.2}. In the subsequent sections, we study two cases in detail: When $m$ and $n$ are both odd (the next section), or both even, Section \ref{even}.

\section{Proof of Theorem \ref{th1.0}}\label{pth1}

Let $m$, $n$ be positive odd integers. Defining the functions
\begin{subequations}
 \label{xieta}
 \begin{align}
  \xi(t)&=\cos \bigl(n \sin^{-1}\!\sqrt{t}\bigr)\cosh \bigl(m \sinh^{-1}\!\sqrt{t}\bigr),\label{xi}\\
    \eta(t)&=\sin \bigl(n \sin^{-1}\!\sqrt{t}\bigr)\sinh \bigl(m \sinh^{-1}\!\sqrt{t}\bigr)\label{eta},
 \end{align}
\end{subequations}
we find from $\eqref{h}$ that
\begin{equation}\label{r1}
    \sqrt{2}\bigl\{i\xi(t)+\eta(t)\bigr\}=e^{i (n+m)\theta/2}\,\overline{h\bigl(e^{i\theta}\bigr)}.
\end{equation}
Using \eqref{pk} we find two orthonormal polynomials for the weight function $\bigl\{\rho(t)\sqrt{1-t^2}\bigr\}^{-1}$
\begin{subequations}\label{polynomials}
    \begin{align}
    &p_{\frac{m+n}{2}}(t)=2\pi^{-1/2}\eta(t)=2\pi^{-1/2}\,\sin \bigl(n \sin^{-1}\!\sqrt{t}\bigr)\sinh \bigl(m \sinh^{-1}\!\sqrt{t}\bigr),\label{kth_polynomial}\\
    &p_{\frac{m+n}{2}+1}(t)=2\pi^{-1/2}\bigl\{t\eta(t)-\sqrt{1-t^2}\,\xi(t)\bigr\}.\label{k_plus_1}
\end{align}
\end{subequations}
This settles $j\ge 0$ in both equations \eqref{f1} and \eqref{f2}. To deal with $j=-1$ in \eqref{f1}, we will use the kernel polynomials (\cite{chihara}, Chapter I, eq. 4.11) defined as
\begin{equation}\label{kernel}
    K_k(t,u)=\sum_{j=0}^kp_j(t)p_j(u)=\frac{\varkappa_k}{\varkappa_{k+1}}\frac{p_{k+1}(t)p_k(u)-p_k(t)p_{k+1}(u)}{t-u},
\end{equation}
where $\varkappa_j$ is the leading coefficient of $p_j(t)$. Since $p_{\frac{m+n}{2}}(0)=0$, this simplifies to
\[
K_{\frac{m+n}{2}}(t,0)=-\varkappa_{\frac{m+n}{2}}\varkappa_{\frac{m+n}{2}+1}^{-1}\,p_{\frac{m+n}{2}+1}(0)\,p_{\frac{m+n}{2}}(t)/t.
\]
The values of the constants in this formula can be worked out from \eqref{xieta} and \eqref{polynomials}:
\begin{equation*}
    \varkappa_{\frac{m+n}{2}+1}=2\varkappa_{\frac{m+n}{2}},\qquad p_{\frac{m+n}{2}+1}(0)=-2\pi^{-1/2}.
\end{equation*}
To finish the proof, we use the reproducing property of the kernel polynomials with $k=(m+n)/{2}$
\[
\pushQED{\qed} 
\int\limits_{-1}^1K_{k}(t,0)\,\frac{dt}{\rho(t)\sqrt{1-t^2}}=1.\qedhere
\popQED
\]

There is a more straightforward way to derive the results of sections \ref{function} and \ref{pth1}, and that also explains generalization considered in the next section. Let $\rho(t)$ be a polynomial of degree $l$ positive on $[-a,b]$. According to Markov-Lukacz theorem, for all $k>\lfloor\frac{n}{2}\rfloor$ there are the unique representations
\begin{align}
    \rho(t)&=p^2_k(t)+(b-t)(a+t)q^2_{k-1}(t)\label{ML1}\\
    &=(b-t)r^2_k(t)+(a+t)s_k^2(t),\label{ML2}
\end{align}
where $p_k(t),q_{k-1}(t),r_k(t),s_k(t)$ are polynomials (of degrees indicated by the subscripts), such that the roots of $p_k(t)$ and $q_{k-1}(t)$, respectively, of $r_k(t)$ and $s_k(t)$, interlace. Then (\cite{bernstein}, pp. 452-467; \cite{imhof}, Theorem A.1.), $p_k$ is the Bernstein-Szeg\"o orthogonal polynomial with respect to the weight function $\bigl\{\rho(t)\sqrt{(b-t)(a+t)}\bigr\}^{-1}$; $q_{k-1}$ with respect to the weight function $\sqrt{(b-t)(a+t)}\,\{\rho(t)\}^{-1}$; $r_k(t)$ with respect to the weight function $\sqrt{(b-t)/(a+t)}\,\{\rho(t)\}^{-1}$, and similarly for $s_k$. The representations $\rho(t)=2\bigl\{\xi^2(t)+\eta^2(t)\bigr\}$ of sections \ref{function} and \ref{a} are exactly either of the form \eqref{ML1} or \eqref{ML2} for suitable $k$, depending on the parities of the numbers $n$ and $m$.

\section{Extension to the interval  \texorpdfstring{$[-a,1],~a>0$}{[-a,1], a>0}}\label{a}

Theorem \ref{th1.0} can be generalized. Let $t\in [-a,1]$ and define
\begin{equation}\label{rho_a_t}
    \rho_a(t)=\cos \bigl(2 n \sin^{-1}\!\sqrt{t}\bigr)+\cosh \bigl(2 m \sinh^{-1}\!\sqrt{t/a}\bigr).
\end{equation}
It is known that the substitution
\[
t=\frac{1}{2}\bigl\{1-a+(1+a)\cos\theta\bigr\},\quad \theta\in[0,\pi],
\]
simplifies the square root expression $\sqrt{(1-t)(a+t)}\,=\frac{1}{2}(1+a)\sin\theta$ so that
\[
\frac{dt}{\sqrt{(1-t)(a+t)}}=-d\theta.
\]
Then $\rho_a(t)$ becomes a polynomial in $\cos\theta$:
\[
\rho_a(t)=\rho(\cos\theta)=T_n\bigl(a-(1+a)\cos\theta\bigr)+T_m\bigl(a^{-1}+(1+a^{-1})\cos\theta\bigr).
\]
After some tedious but quite straightforward algebra we obtain the representation
\[
\sqrt{2}\,\cos \bigl( n \sin^{-1}\sqrt{t}-im \sinh^{-1}\sqrt{t/a}\bigr)=i^n\,e^{-i (n+m)\theta/2}\,h_a\bigl(e^{i\theta}\bigr),
\]
\begin{align}\label{ha}
    \nonumber h_a(z)=\frac{1}{\sqrt{2}}\left(\frac{a+1}{2 \sqrt{a}}\right)^m\bigg\{&\left(\sqrt{z^2+2cz+1}+1+cz\right)^{\frac{m+n}{2}} \left(\sqrt{z^2+2cz+1}+c+z\right)^{\frac{m-n}{2}}\\
    +(-1)^n&\left(\sqrt{z^2+2cz+1}-1-cz\right)^{\frac{m+n}{2}} \left(\sqrt{z^2+2cz+1}-c-z\right)^{\frac{m-n}{2}}\bigg\},
\end{align}
where $c=\frac{1-a}{1+a}$. $h_a(z)$ is a polynomial in $z$. The above is a generalization of the representation \eqref{h}, which corresponds to $a=1$. One can show that this representation satisfies the conditions of Lemma \ref{lemma}. Thus the normalized representation required by Theorem \ref{sb} is
\[
{\rho}(\cos\theta)=\rho_a(t)=|h_a(e^{i\theta})|^2.
\]

It should be noted that consideration of the more general interval $[-a,b]$, where $b>0$, along with the substitution $t=\frac{1}{2}\big\{b-a+(b+a)\cos\theta\big\}$ does not lead to anything essentially new.

The formulas above are valid for all positive integers. In the following, we specify $n$, $m$ to be positive odd integers. Then
\begin{equation}\label{r2}
    \sqrt{2}\bigl\{i\xi_a(t)+\eta_a(t)\bigr\}=e^{i (n+m)\theta/2}\,\overline{h_a\bigl(e^{i\theta}\bigr)},
\end{equation}
where
\[
\xi_a(t)=\cos \bigl(n \sin^{-1}\!\sqrt{t}\bigr)\cosh \bigl(m \sinh^{-1}\!\sqrt{t/a}\bigr),
\]
\[
\eta_a(t)=\sin \bigl(n \sin^{-1}\!\sqrt{t}\bigr)\sinh \bigl(m \sinh^{-1}\!\sqrt{t/a}\bigr).
\]
The two orthonormal polynomials corresponding to the weight function $\bigl\{\rho_a(t)\sqrt{(1-t)(a+t)} \bigr\}^{-1}$ are
\begin{equation}\label{pka}
    p_{\frac{m+n}{2}}(t)=2\pi^{-1/2}\,\eta_a(t)=2\pi^{-1/2}\,\sin \bigl(n \sin^{-1}\!\sqrt{t}\bigr)\sinh \bigl(m \sinh^{-1}\!\sqrt{t/a}\bigr),
\end{equation}
\begin{equation}\label{pk1a}
    p_{\frac{m+n}{2}+1}(t)=\frac{2t-1+a}{1+a}\,p_{\frac{m+n}{2}}(t)-\frac{4\pi^{-1/2}\,}{1+a}\sqrt{(1-t)(a+t)}\,\xi_a(t),
\end{equation}
by \eqref{pk}. The ratio of their leading coefficients is
\begin{equation}\label{ratio}
    \varkappa_{\frac{m+n}{2}+1}/\varkappa_{\frac{m+n}{2}}=\frac{4}{1+a}.
\end{equation}
\begin{Th}[{1*}]
Let $n$ and $m$ be positive odd integers and $a>0$. Then
\begin{equation}\label{f1*}
\int\limits_{-a}^{1}\frac{\sin \bigl(n \sin^{-1}\sqrt{t}\bigr)\sinh \bigl(m \sinh^{-1}\sqrt{t/a}\bigr)}{\cos \bigl( 2 n \sin^{-1}\sqrt{t}\bigr)+\cosh \bigl(2 m \sinh^{-1}\sqrt{t/a}\bigr)}\frac{t^j\,dt}{ \sqrt{(1-t)(1+t/a)}}=\begin{cases}
        \pi/2,\quad &j=-1,\\
        0,\quad &j=0,1,\ldots,\frac{m+n-2}{2},
        \end{cases}
\end{equation}
\begin{equation*}
    \int\limits_{-a}^{1}\frac{\cos \bigl(n \sin^{-1}\!\sqrt{t}\bigr)\cosh \bigl(m \sinh^{-1}\!\sqrt{t/a}\bigr)}{\cos \bigl(2 n \sin^{-1}\!\sqrt{t}\bigr)+\cosh \bigl(2 m \sinh^{-1}\!\sqrt{t/a}\bigr)}\,t^jdt=0,\quad j=0,1,\ldots,\tfrac{m+n-4}{2}.
\end{equation*}
\end{Th}
It should be noted that the symmetric case $m=n$ of \eqref{f1*} also follows from \eqref{form1} and the identity $\tan^{-1}(a)+\tan^{-1}(1/a)={\pi}/{2}$. 

Equation \eqref{f1*} has an additional (integer) parameter $m$ compared to \eqref{form1}. However, the integration range now covers the entire interval $[-a,1]$. There does not seem to be a closed-form evaluation of the integral in \eqref{f1*} when the integration range is $[0,1]$ (in other words, there do not seem to be any non-trivial extensions of \eqref{form1} that include an additional parameter).

\section{Even integers \texorpdfstring{$m$}{m}, \texorpdfstring{$n$}{n}}\label{even}
Again, one can restrict consideration to $m\ge n$. Then, the degree of the polynomial $\rho_a(t)$ \eqref{rho_a_t} is $m$. Taking into account that $n$ and $m$ are even we get with $\xi_a$, $\eta_a$ defined in the previous section
\[
-\sqrt{2}\bigl\{\xi_a(t)-i\eta_a(t)\bigr\}=e^{i (n+m)\theta/2}\,\overline{h_a\bigl(e^{i\theta}\bigr)},
\]
where $h_a(z)$ is given by \eqref{ha}. The difference from \eqref{r2} is the phase factor of $-1$ instead of $i$.  The orthonormal polynomials of interest are
\[
p_{\frac{m+n}{2}}(t)=2\pi^{-{1}/{2}}\,\xi_a(t)=2\pi^{-{1}/{2}}\cos \bigl(n \sin^{-1}\!\sqrt{t}\bigr)\cosh \bigl(m \sinh^{-1}\!\sqrt{t/a}\bigr),
\]
\[
p_{\frac{m+n}{2}+1}(t)=\frac{2 t-1+a}{1+a}\,p_{\frac{m+n}{2}}(t)+\frac{4}{1+a}\,\sqrt{\frac{1}{\pi}}\,\sqrt{(1-t)(a+t)}\,\eta_a(t).
\]
Since $p_{\frac{m+n}{2}+1}(0)=0$, equation \eqref{kernel} simplifies to
\[
K_{\frac{m+n}{2}}(t,0)=\varkappa_{\frac{m+n}{2}}\varkappa_{\frac{m+n}{2}+1}^{-1}\,p_{\frac{m+n}{2}}(0)\,p_{\frac{m+n}{2}+1}(t)/t,
\]
where
\[
\varkappa_{\frac{m+n}{2}+1}/\varkappa_{\frac{m+n}{2}}=\frac{4}{1+a},\qquad p_{\frac{m+n}{2}}(0)=\frac{4}{1+a}\sqrt{\frac{1}{\pi}}.
\]
The resulting theorem is: 
\begin{Th} Let $n$ and $m$ be positive even integers and $a>0$. Then
\begin{equation}\label{sinsinh2}
    \int\limits_{-a}^{1}\frac{\sin \bigl(n \sin^{-1}\!\sqrt{t}\bigr)\sinh \bigl(m \sinh^{-1}\!\sqrt{t/a}\bigr)}{\cos \bigl(2 n \sin^{-1}\!\sqrt{t}\bigr)+\cosh \bigl(2 m \sinh^{-1}\!\sqrt{t/a}\bigr)}\,t^jdt=\begin{cases}
        \pi/2,\quad &j=-1 \\
        0,\quad &j=0,1,\ldots,\frac{m+n-4}{2},
    \end{cases}
\end{equation}
\begin{equation}\label{coscosheven2}
    \int\limits_{-a}^{1}\frac{\cos \bigl(n \sin^{-1}\!\sqrt{t}\bigr)\cosh \bigl(m \sinh^{-1}\!\sqrt{t/a}\bigr)}{\cos \bigl(2 n \sin^{-1}\!\sqrt{t}\bigr)+\cosh \bigl(2 m \sinh^{-1}\!\sqrt{t/a}\bigr)}\,\frac{t^jdt}{\sqrt{(1-t)(a+t)}}=0,\quad j=0,1,\ldots,\tfrac{m+n-2}{2}.
\end{equation}
\end{Th}
\eqref{coscosheven2} is a two-parameter generalization of \eqref{coscosheven}. Equation \eqref{sinsinh2} have the following interpretation: The degree $\tfrac{m+n-2}{2}$ orthogonal polynomial for the weight function $\sqrt{(1-t)(a+t)}\,\{\rho_a(t)\}^{-1}$ is
\[
\frac{\sin \bigl(n \sin^{-1}\!\sqrt{t}\bigr)\sinh \bigl(m \sinh^{-1}\!\sqrt{t/a}\bigr)}{\sqrt{(1-t)(a+t)}}.
\]
An alternative way to derive this is to start directly from equation \eqref{pk1} instead of \eqref{pk}.

Similarly, it follows from \eqref{pk2} that when $n$ is even and $m$ is odd, the degree $\tfrac{m+n-1}{2}$ orthogonal polynomial for the weight function $\sqrt{(1-t)/(a+t)}\,\{\rho_a(t)\}^{-1}$ is
\[
\frac{\sin \bigl(n \sin^{-1}\!\sqrt{t}\bigr)\sinh \bigl(m \sinh^{-1}\!\sqrt{t/a}\bigr)}{\sqrt{1-t}}.
\]

\section{Some other theorems}\label{d}
Expressions like \eqref{rho_a_t} with different integer parameters can be used as building blocks for more complicated weight functions. Here we restrict our attention to the simplest of such functions
\[
\widetilde{\rho}_a(t)=\bigl\{\rho_a(t)\bigr\}^2=\bigl\{\cos \bigl(2n \sin^{-1}\!\sqrt{t}\bigr)+\cosh \bigl(2m \sinh^{-1}\!\sqrt{t/a}\bigr)\bigr\}^2,
\]
where $m,n$ are positive integers. From \eqref{r1}, we readily obtain the representation
\[
2\bigl\{\xi_a(t)-i\eta_a(t)\bigr\}^2=e^{i (n+m)\theta}\,\overline{h_a^2\bigl(e^{i\theta}\bigr)},
\]
with $h_a(z)$ defined in \eqref{ha}. Applying \eqref{pk1} we obtain two orthonormal polynomials for the weight function $\sqrt{(1-t)(a+t)}\,\{\widetilde{\rho}_a(t)\}^{-1}$:
\[
p_{m+n-1}(t)=\sqrt{\frac{2}{\pi}}\,\frac{\sin \bigl(2n \sin^{-1}\!\sqrt{t}\bigr)\sinh \bigl(2m \sinh^{-1}\!\sqrt{t/a}\bigr)}{\sqrt{(1-t)(a+t)}},
\]
\[
p_{m+n}(t)=\frac{2 t-1+a}{1+a}\,p_{m+n-1}(t)-\frac{4}{1+a}\,\sqrt{\frac{2}{\pi}}\,\bigl\{\xi_a^2(t)-\eta_a^2(t)\bigr\}.
\]
Since $p_{m+n-1}(0)=0$, \eqref{kernel} simplifies to
\[
K_{m+n-1}(t,0)=-\varkappa_{m+n-1}\varkappa_{m+n}^{-1}\,p_{m+n}(0)\,p_{m+n-1}(t)/t,
\]
where
\[
\varkappa_{m+n}/\varkappa_{m+n-1}=\frac{4}{1+a},\qquad p_{m+n}(0)=-\frac{4}{1+a}\,\sqrt{\frac{2}{\pi}}.
\]
Thus, we obtain:
\begin{Th}  Let $n$ and $m$ be positive integers and $a>0$. Then
\begin{equation}\label{square}
    \int\limits_{-a}^{1}\frac{\sin \bigl(2n \sin^{-1}\!\sqrt{t}\bigr)\sinh \bigl(2m \sinh^{-1}\!\sqrt{t/a}\bigr)}{\bigl\{\cos \bigl(2n \sin^{-1}\!\sqrt{t}\bigr)+\cosh \bigl(2m \sinh^{-1}\!\sqrt{t/a}\bigr)\bigr\}^2}\,t^jdt=\begin{cases}
        \pi/2,\quad j=-1,\\
        0,\quad j=0,1,\ldots,{n+m}-2.
    \end{cases}
\end{equation}
\end{Th}
\eqref{square} is closely related to the integral
\begin{equation}\label{alpha}
    \int\limits_0^\infty\frac{\sin(x\sin \alpha)\sinh(x \cos \alpha)}{\left\{\cosh(x\cos \alpha)+\cos (x\sin \alpha)\right\}^2}\frac{dx}{x}=\frac{\alpha}{2},
\end{equation}
mentioned in Section 7 of \cite{nicholson}. Integrals similar to \eqref{alpha} were also studied in \cite{chowla}.

\section{Gauss quadratures}\label{gq}

\begin{Th}\label{quad1} Let $n$ and $m$ be positive odd integers, $a>0$, and define 
\begin{equation}\label{ab}
    \alpha_z=2n\sinh^{-1}\Bigl(a^{-1/2}\sin\frac{\pi z}{2n}\Bigr),\quad \beta_z=2m\sinh^{-1}\Bigl(a^{1/2}\sin\frac{\pi z}{2m}\Bigr).
\end{equation}
Then for any polynomial $p(t)$ of degree at most $m+n-1$
\begin{align}\label{quadrature}
    \nonumber\int\limits_{-a}^{1}&\frac{p(t)}{\cos \bigl(2 n \sin^{-1}\!\sqrt{t}\bigr)+\cosh \bigl(2 m \sinh^{-1}\!\sqrt{t/a}\bigr)}\frac{dt}{\sqrt{(1-t)(a+t)}}\\
    &=\frac{\pi }{2mn}\,p(0)+\frac{2 \pi }{n}\sum_{i=1}^{n/2}\frac{\tanh\frac{\alpha_{2i}}{2n}}{\sinh\bigl(\frac{m}{n}\alpha_{2i}\bigr)}\,p\bigl(\sin^2\tfrac{\pi i}{n}\bigr)+\frac{2 \pi }{m}\sum_{j=1}^{m/2}\frac{\tanh\frac{\beta_{2j}}{2m}}{\sinh\bigl(\frac{n}{m}\beta_{2j}\bigr)}\,p\bigl(-a\sin^2\tfrac{\pi j}{m}\bigr).
\end{align}
\end{Th}
\begin{proof} According to \eqref{pka}, the degree $k=\frac{m+n}{2}$ orthonormal polynomial corresponding to the weight function $\bigl\{\rho_a(t)\sqrt{(1-t)(a+t)}\bigr\}^{-1}$, where $\rho_a(t)$ is as in \eqref{rho_a_t}, is
\[
p_{k}(t)=2\pi^{-1/2}\,\sin \bigl(n \sin^{-1}\!\sqrt{t}\bigr)\sinh \bigl(m \sinh^{-1}\!\sqrt{t/a}\bigr).
\]
Its $k$ roots $x_s$ are
\[
0;\quad \sin^2\frac{\pi i}{n},\quad i=1,2,\ldots,\tfrac{n-1}{2};\quad -a\sin^2\frac{\pi j}{m},\quad j=1,2,\ldots,\tfrac{m-1}{2}.
\]
Gauss quadrature formula \cite{chihara} now takes the form 
\[
\int\limits_{-a}^{1}p(t)\frac{dt}{\rho_a(t)\sqrt{(1-t)(a+t)}}=\sum_{s=1}^{k}p(x_s)w_s,
\]
where different representations for the weights (obtained using \eqref{kernel}) are 
\begin{equation}\label{weights}
    w_s=\biggl\{\sum_{r=0}^{k-1}p_r^2(x_s)\biggr\}^{-1}=\biggl\{\sum_{r=0}^{k}p_r^2(x_s)\biggr\}^{-1}=\frac{\varkappa_{k}}{\varkappa_{k-1}p_{k-1}(x_s)p_k' (x_s)}=\frac{-\varkappa_{k+1}}{\varkappa_kp_{k+1}(x_s)p_k' (x_s)}.
\end{equation}
To calculate $w_s$, we use formulas \eqref{pka},\eqref{pk1a},\eqref{ratio} and the fourth representation in \eqref{weights}.
\end{proof}

Theorem \ref{quad1} can be extended to a pair of positive even integers using results of Section \ref{even}, or to integers of opposite parities (see the remark at the end of Section \ref{even}). The limiting $n,m\to\infty$ form of the integral in Theorem \ref{quad1} is the same regardless of the parity of $n$ and $m$. This means that we have established an analog of Theorem 4.5 from \cite{ismail}, which loosely speaking gives two discrete measures that have the same moments as the continuous measure \eqref{iv} (the corresponding moment problem is known to be indeterminate). A similar situation is also encountered in Section \ref{appendix}.

For the purpose of demonstrating some other possibilities, consider the weight functions of Section \ref{d}.

\begin{Th} Let $n$ and $m$ be positive integers, $a>0$, and $\alpha_z$, $\beta_z$ be defined according to \eqref{ab}. Then for any polynomial $p(t)$ of degree at most $2m+2n-3$
    \begin{align}
    \nonumber\frac{2mn}{\pi a}&\int\limits_{-a}^{1}\frac{p(t)\sqrt{(1-t)(a+t)}}{\big\{\!\cos \bigl(2 n \sin^{-1}\!\sqrt{t}\bigr)+\cosh \bigl(2 m \sinh^{-1}\!\sqrt{t/a}\bigr)\big\}^2}\, dt\\
    \nonumber&=\frac{p(0)}{4}+\sum_{i=1}^{n-1}\frac{m\sinh\frac{\alpha_{i}}{n}}{\sinh\bigl(\frac{m}{n}\alpha_{i}\bigr)}\,\frac{\cos^2\frac{\pi i}{2n}\cdot p\bigl(\sin^2\frac{\pi i}{2n}\bigr)}{\cosh\bigl(\frac{m}{n}\alpha_{i}\bigr)+(-1)^i}+\sum_{j=1}^{m-1}\frac{n\sinh\frac{\beta_{j}}{m}}{\sinh\bigl(\frac{n}{m}\beta_{j}\bigr)}\,\frac{\cos^2\frac{\pi j}{2m}\cdot p\bigl(-a\sin^2\frac{\pi j}{2m}\bigr)}{\cosh\bigl(\frac{n}{m}\beta_{j}\bigr)+(-1)^j}.
\end{align}
\end{Th}

\begin{proof} The $m+n-1$ roots of the polynomials $p_{m+n-1}(t)$ from the previous section are
\[
0;\quad \sin^2\frac{\pi i}{2n},\quad i=1,2,\ldots,{n-1};\quad -a\sin^2\frac{\pi j}{2m},\quad j=1,2,\ldots,m-1.
\]
To finish the proof, we apply Gauss quadrature formula. \end{proof}

The more general weight functions $\sqrt{(1-t)(a+t)}\,\{\widetilde{\rho}_a(t)\}^{-1}$ with 
\[
\widetilde{\rho}_a(t)=\bigl\{\cos \bigl(2n \sin^{-1}\!\sqrt{t}\bigr)+\cosh \bigl(2m \sinh^{-1}\!\sqrt{t/a}\bigr)\bigr\}\bigl\{\cos \bigl(2n \sin^{-1}\!\sqrt{t}\bigr)+\cosh \bigl(2m' \sinh^{-1}\!\sqrt{t/a}\bigr)\bigr\},
\]
also admit orthogonal polynomials (possibly modified by an additional factor of $\sqrt{a+t}$ depending on whether $m$ and $m'$ are of the same parity or not) with known roots:
\[
\frac{\sin \bigl(2n \sin^{-1}\!\sqrt{t}\bigr)\sinh \bigl((m+m') \sinh^{-1}\!\sqrt{t/a}\bigr)}{\sqrt{1-t}}.
\]
We only state the limiting form of the corresponding quadrature formula:
\begin{Th} Let $\alpha$ and $\beta$ be positive real numbers, and let $p(x)$ be a polynomial. Then
\begin{align*}
    &\int\limits_{\mathbb{R}} \frac{p(x)\,dx}{\big\{\!\cos \bigl(\sqrt{x}\bigr)+\cosh \bigl(\alpha\sqrt{x}\bigr)\big\}\big\{\!\cos \bigl(\sqrt{x}\bigr)+\cosh \bigl(\beta\sqrt{x}\bigr)\big\}}=\frac{\pi p(0)}{\alpha+\beta}\\
    &+\sum_{j=1}^\infty\frac{2\pi^2 j}{\sinh\frac{\pi(\alpha+\beta)j}{2}}\,\frac{p\bigl(\pi^2j^2\bigr)}{\cosh\frac{\pi(\alpha+\beta)j}{2}+(-1)^j\cosh\frac{\pi(\alpha-\beta)j}{2}}+\frac{8\pi^2}{(\alpha+\beta)^2}\sum_{j=1}^\infty\frac{j}{\sinh\frac{2\pi j}{\alpha+\beta}}\,\frac{p\left(-\frac{4\pi^2j^2}{(\alpha+\beta)^2}\right)}{\cosh\frac{2\pi j}{\alpha+\beta}+\cos\frac{2\pi\alpha j}{\alpha+\beta}}.
\end{align*} 
\end{Th}

Quadrature formulas can also be studied when the roots of the orthogonal polynomials are not known explicitly.  For example, when the limiting behaviour of the roots as $n,m\to\infty$ is described by the transcendental equation $\tan(x)\tanh(x)=-1$ \cite{berg}.

As a consequence of the considerations outlined above (see also the next section):
\begin{Cor} For any positive integer $n$
\[
    \int\limits_{-a}^1\frac{\sqrt{(1-t)(a+t)}}{a+2t+aT_n(1-2t)}\, dt=\frac{\pi}{4},\quad a>0,
\]
\[
    \int\limits_{-1}^1\frac{1}{1+2t+T_n(1-2t)}\frac{dt}{\sqrt{1-t^2} }=\frac{\pi}{\sqrt{8}}\frac{\left(\sqrt{2}+1\right)^{2n}+1}{\left(\sqrt{2}+1\right)^{2n}-1},
\]
and positive integers $n$ and $m$ of the same parity
\[
\int\limits_{-1}^1\frac{\sqrt{1-t^2}}{\{1+2t+T_n(1-2t)\}\{1+2t+T_m(1-2t)\}}\, dt=\frac{\pi}{4 (n+m)}\sum _{|j|<\frac{n+m}{2}} \frac{1+\cos \frac{2 \pi  j}{n+m}}{2-\cos\frac{2 \pi  j}{n+m}+\cos\frac{2\pi m j }{n+m}}.
\]
\end{Cor}
In general, the sum in the last formula does not appear to have a simple closed-form evaluation.

\section{Finite analogs of generating functions from Kuznetsov's paper \texorpdfstring{\cite{kuznetsov}}{}}\label{section_generating_functions}

We are going to rewrite the right hand side of \eqref{quadrature} in Theorem \ref{quad1} as a single sum.

\begin{Th}\label{generating_functions}
Let $n$, $m$, $u$  be integers such that $|u|<n$. Let $a>0$ and define $\alpha_z$ as in \eqref{ab}. Then
\begin{align}\label{integraltosummation1}
    \nonumber\int\limits_{-a}^1&\frac{\cos \bigl(2 u \sin ^{-1}\sqrt{t}\bigr)}{\cos \bigl(2 n \sin ^{-1}\sqrt{t}\bigr)+\cosh \bigl(2 m \sinh ^{-1}\sqrt{t/a}\bigr)}\frac{dt}{\sqrt{(1-t)(a+t)}}\\&{\phantom{................................}}=\frac{\pi}{2n}\sum_{j=1}^{2n}\frac{(-1)^{j-1}}{\coth\frac{\alpha_j}{2n}}\left\{\tanh\frac{m\alpha_j}{2 n}\right\} ^{(-1)^j}\!\!\!\!\cdot\cos\frac{\pi j u}{n},
\end{align}
and, more generally, for any polynomial $p(t)$ of degree less than $n$
\begin{align}\label{integraltosummation2}
    \nonumber\int\limits_{-a}^1&\frac{p(t)}{\cos \bigl(2 n \sin ^{-1}\sqrt{t}\bigr)+\cosh \bigl(2 m \sinh ^{-1}\sqrt{t/a}\bigr)}\frac{dt}{\sqrt{(1-t)(a+t)}}\\&{\phantom{................................}}=\frac{\pi}{2n}\sum_{j=1}^{2n}\frac{(-1)^{j-1}}{\coth\frac{\alpha_j}{2n}}\left\{\tanh\frac{m\alpha_j}{2 n}\right\} ^{(-1)^j}\!\!\!\!\cdot p\Bigl(\sin^2\frac{\pi j}{2n}\Bigr).
\end{align}
\end{Th}
\begin{proof} Here, we prove the theorem only for an odd pair of integers $n$, $m$. Other alternatives can be proved in a similar way starting from appropriate variants of Gauss quadrature formula. 

Since \eqref{integraltosummation2} is a consequence of \eqref{integraltosummation1} (any polynomial can be written as a linear combination of Chebyshev polynomials of the first kind), it is enough to prove \eqref{integraltosummation1}.

Equation \eqref{quadrature} with degree $u$ polynomial $p(t)=\cos \bigl(2 u \sin ^{-1}\sqrt{t}\bigr)$ yields for the value of the integral
\[
\frac{\pi }{2mn}+\frac{2 \pi }{n}\sum_{i=1}^{n/2}\frac{\tanh\frac{\alpha_{2i}}{2n}}{\sinh\bigl(\frac{m}{n}\alpha_{2i}\bigr)}\cos\frac{2\pi i u}{n}+\frac{2 \pi }{m}\sum_{j=1}^{m/2}\frac{\tanh\frac{\beta_{2j}}{2m}}{\sinh\bigl(\frac{n}{m}\beta_{2j}\bigr)}\cosh\bigl(\tfrac{u}{m}\beta_{2j}\bigr).
\]
We transform the first sum according to ${2\sum_{i=1}^{n/2}=\sum_{i=1}^{n-1}}$ using the fact that the summand is symmetric under $i\to n-i$. In the second sum, we rewrite the summand using the identity valid for integers $|u|<n$
\[
\frac{\tanh z}{\sinh (2 n z)}\cosh (2 u z)=\frac{1}{2n}\sum_{i=1}^{2n-1}(-1)^{i-1}\frac{\sin^2\frac{\pi i}{2n}}{\sinh^2z+\sin^2\frac{\pi i}{2n}}\cos\frac{\pi i u}{n},
\]
with $z=2m\beta_{2j}=\sinh^{-1}\Bigl(\sin\frac{\pi j}{m}\Bigr)$:
\[
\frac{\pi }{mn}\sum_{j=1}^{m/2}\sum_{i=1}^{2n-1}(-1)^{i-1}\frac{\sin^2\frac{\pi i}{2n}}{\sin^2\frac{\pi j}{m}+\sin^2\frac{\pi i}{2n}}\cos\frac{\pi i u}{n}.
\]
For odd $m$, the sum over $j$ is
\[
\sum _{j=1}^{m/2} \frac{\sinh ^2z}{\sin ^2\frac{\pi j}{m}+\sinh ^2z}=\frac{m \tanh z}{2\tanh (mz)}-\frac{1}{2}.
\]
After these transformations, the value of the integral becomes
\[
\frac{\pi }{2mn}+\frac{\pi }{n}\sum_{i=1}^{n-1}\frac{\tanh\frac{\alpha_{2i}}{2n}}{\sinh\bigl(\frac{m}{n}\alpha_{2i}\bigr)}\cos\frac{2\pi i u}{n}+\frac{\pi }{2n}\sum_{i=1}^{2n-1}(-1)^{i-1}\frac{\tanh\frac{\alpha_{i}}{2n}}{\tanh\bigl(\frac{m}{2n}\alpha_{i}\bigr)}\cos\frac{\pi i u}{n}-\frac{\pi }{2mn}\sum_{i=1}^{2n-1}(-1)^{i-1}\cos\frac{\pi i u}{n}.
\]
The last sum cancels the first term due to the simple identity valid for integers $|u|<n$
\[
\sum_{i=1}^{2n-1}(-1)^{i-1}\cos\frac{\pi i u}{n}=1.
\]
Next, we split the second sum into even and odd terms, and then combine even terms with the first sum using $~2/\sinh(2x)-\coth(x)=-\tanh(x)$.
\end{proof}

There is a transformation for the right hand side of the formula \eqref{integraltosummation1} of Theorem \ref{generating_functions}
\[
\frac{\pi}{2m}\sum_{j=1}^{2m}\frac{(-1)^{j-1}}{\coth\frac{\beta_j}{2m}}\left\{\tanh\frac{n\beta_j}{2 m}\right\} ^{(-1)^j}\!\!\!\!\cdot\cosh\frac{u\beta_j}{m},\qquad |u|<m,
\]
analogous to imaginary transformation of Jacobi's elliptic functions.

Theorem \ref{generating_functions} is a non-symmetric ($m\neq n$, $a\neq 1$) extention of Theorem 4 from \cite{nicholson}. The proof of Theorem 4 in \cite{nicholson} used ad hoc methods that could not be applied to more general integrals. The case $u=0$ of Theorem \ref{generating_functions} is a finite analog of Ismail and Valent's integral \cite{ismail}. There are also other integrals with $\cos \left(2 u \sin ^{-1}\sqrt{t}\right)$ replaced by $\sin \left(2 u \sin ^{-1}\sqrt{t}\right)/\sqrt{t(1-t)}$, or with the differences $\cosh \bigl(2 m \sinh ^{-1}\sqrt{t/a}\bigr)-\cos \bigl(2 n \sin ^{-1}\sqrt{t}\bigr)$ in the denominator (Section \ref{appendix}), etc. 

\section{Constructing measures on \texorpdfstring{$\mathbb{R}$}{R}}\label{contructing_measures}

Let $\mu$ be a positive measure on $\mathbb{R}$ with infinitely many points in its support, and with finite moments $\int x^jd\mu$,~ $j=0,1,2,\ldots$ . Let
\[
p_k(x)=\varkappa_kx^k+\ldots,\qquad \varkappa_k>0,
\]
denote the corresponding orthonormal polynomials with positive leading coefficients. The orthonormality conditions are
\[
\int p_i(t)p_j(t)d\mu(t)=\delta_{ij},\quad i,j=0,1,2,\ldots .
\]
It was proved in \cite{lubinsky2},\cite{lubinsky}, that: If $\operatorname{Im}\gamma\neq 0$, then for any polynomial $p(x)$ of degree at most $2k-2$
\begin{equation}\label{measure1}
    \int\limits_{\mathbb{R}}p(x)\frac{|\operatorname{Im}\gamma|/\pi}{|\gamma p_k(x)-p_{k-1}(x)|^2}\, dx=\frac{\varkappa_{k-1}}{\varkappa_k}\int p(t)d\mu(t),
\end{equation}
\begin{equation}\label{measure4}
    \int\limits_{\mathbb{R}}p(x)\frac{|\operatorname{Im}\gamma|/\pi}{|p_k(x)-\gamma p_{k-1}(x)|^2}\, dx=\frac{\varkappa_{k-1}}{\varkappa_k}\int p(t)d\mu(t).
\end{equation}
A closely related formula can also be found in the book \cite{atkinson}, Theorem 4.9.1.

Considerations given in the paper \cite{kuznetsov2} suggest that the constant $\gamma$ in \eqref{measure1} and \eqref{measure4} could be replaced by a Pick function $\phi(x)$, i.e. a function analytic in the upper half plane 
\[
\mathbb{H^+}=\{z\in\mathbb{Z}: \operatorname{Im}z>0\}
\]
and that maps upper half plane into itself: $\phi(\mathbb{H^+})\subseteq \mathbb{H^+}$ (to be precise, a more restricted class of functions from Definition 2 of \cite{kuznetsov2}).

Proof of this statement in its full generality is beyond the scope of the present paper. Below we prove it only for rational Pick functions of a special type, in which case a proof based on standard contour integration arguments can be given.

\begin{Th}\label{measure3} Let $\mu$, $\{p_j\}$, and $\{\varkappa_j\}$ be defined as above. Let $\phi(x)$ be a rational function of the form
\[
\phi(x)=\beta x+\gamma -\sum\frac{c_r}{x-z_r},\qquad \beta\ge 0,~\operatorname{Im}\gamma>0,~c_r\ge 0,~\operatorname{Im}z_r< 0.
\]
Then for any polynomial $p(x)$ of degree at most $2k-2$
\begin{equation}\label{measure2}
    \int\limits_{\mathbb{R}}p(x)\frac{\operatorname{Im}\phi(x)/\pi}{|\phi(x)p_k(x)-p_{k-1}(x)|^2}\, dx=\frac{\varkappa_{k-1}}{\varkappa_k}\int p(t)d\mu(t),
\end{equation}
\begin{equation}\label{measure5}
    \int\limits_{\mathbb{R}}p(x)\frac{\operatorname{Im}\phi(x)/\pi}{|p_k(x)+\phi(x)p_{k-1}(x)|^2}\, dx=\frac{\varkappa_{k-1}}{\varkappa_k}\int p(t)d\mu(t).
\end{equation}
\end{Th}

\begin{proof} It is sufficient to prove the theorem for a real polynomail $p(x)$. Denote the left hand side of \eqref{measure2} as $I_{\phi}$. Consider the Cauchy principal value (\cite{lubinsky2}, Section 2)
\[
G_{\phi}=\operatorname{P.V.} \int\limits_{\mathbb{R}}f(x)\,dx,\qquad f(x)=\frac{1}{\pi}\frac{p(x)}{p_k(x)}\frac{1}{p_{k-1}(x)-\phi(x)p_k(x)}.
\]
Then due to (recall that the orthonormal polynomials $p_j(x)$ are real (\cite{chihara}, Chapter I.3))
\[
p(x)\frac{\operatorname{Im}\phi(x)/\pi}{|\phi(x)p_k(x)-p_{k-1}(x)|^2}=\operatorname{Im}f(x), \qquad x\in\mathbb{R},
\]
we have $I_{\phi}=\operatorname{Im} G_{\phi}$.

Let $x_1,x_2,\ldots,x_k$, be the $k$ roots of the polynomial $p_k$.

Consider the contour $C$ in $\mathbb{H^+}\cup \mathbb{R}$ composed of: (i) a semicircle $\Gamma_R$ of large radius $R$ centered at the origin, (ii) $k$ semicircles $\Gamma_{\varepsilon,s}$ of small radius $\varepsilon$ centered at $x_s$, (iii) interval $[-R,R]$ with $k$ subintervals $(x_s-\varepsilon,x_s+\varepsilon)$ excluded. 

Because of interlacing property for the zeros of orthogonal polynomials (\cite{chihara}, Chapter I.5)
\[
\frac{p_{k-1}(z)}{p_k(z)}=\sum_{s=1}^k\frac{d_s}{z-x_s},\qquad d_s> 0,~x_s\in\mathbb{R}.
\]
Thus the rational function
\[
\phi(z)-\frac{p_{k-1}(z)}{p_k(z)}
\]
maps $\mathbb{H^+}$ into itself. Moreover, since $p_k$ has only real zeroes, it follows from this that $p_k(x)\bigl(p_{k-1}(z)-\phi(z)p_k(z)\bigr)$ does not have zeroes in $\mathbb{H^+}$, and as a result $f(z)$ is analytic in $\mathbb{H^+}$. Hence, by residue theorem
\[
\int\limits_Cf(z)\,dz=0.
\]

Next, we take the limit $R\to\infty$, $\varepsilon\to 0$. In this limit, the integral over $\Gamma_R$ vanishes due to $\operatorname{deg}p\le 2k-2$, while the integral over $\Gamma_{\varepsilon,s}$ equals $-\pi i\cdot \underset{z=x_s}{\operatorname{res}}f(z)$, and we obtain
\[
G_{\phi}-\pi i \cdot \frac{1}{\pi}\sum_{s=1}^{k}\frac{p(x_s)}{p_{k-1}(x_s)p_k' (x_s)}=0.
\]
Using $I_{\phi}=\operatorname{Im} G_{\phi}$ and Gauss quadrature formula in the form
\[
\int p(t)d\mu(t)=\frac{\varkappa_{k}}{\varkappa_{k-1}}\sum_{s=1}^{k}\frac{p(x_s)}{p_{k-1}(x_s)p_k' (x_s)}
\]
(the third representation in \eqref{weights} for the weights $w_s$) we complete the proof. Proof of \eqref{measure5} is similar. \end{proof}

Note that $\phi(x)$ given in the statement of Theorem \ref{measure3} satisfies $\operatorname{Im}\phi(x)\ge\operatorname{Im}\gamma>0$ for $x\in\mathbb{R}$.

\textit{Example.} Now we apply Theorem \ref{measure3} to the measure considered in Sections \ref{function} and \ref{pth1}
\[
d\mu(t)=\frac{1}{\cos \bigl(2 n \sin^{-1}\!\sqrt{t}\bigr)+\cosh \bigl(2 m \sinh^{-1}\!\sqrt{t}\bigr)}\frac{dt}{\sqrt{1-t^2}},\quad t\in[-1,1],
\]
where $m$ and $n$ are positive odd integers. The corresponding orthonormal polynomials of degrees $k=\frac{m+n}{2}$ and $k-1$ are
\[
p_{k}(t)=2\pi^{-1/2}\eta(t),
\]
\[
p_{k-1}(t)=2\pi^{-1/2}\bigl\{t\eta(t)+\sqrt{1-t^2}\,\xi(t)\bigr\},
\] 
where
\[
\xi(t)=\cos \bigl(n \sin^{-1}\!\sqrt{t}\bigr)\cosh \bigl(m \sinh^{-1}\!\sqrt{t}\bigr),
\]
\[
\eta(t)=\sin \bigl(n \sin^{-1}\!\sqrt{t}\bigr)\sinh \bigl(m \sinh^{-1}\!\sqrt{t}\bigr).
\]
The ratio of the leading coefficients is ${\varkappa_{k}}/{\varkappa_{k-1}}=2$. Note also that
\[
\cos \bigl(2 n \sin^{-1}\!\sqrt{t}\bigr)+\cosh \bigl(2 m \sinh^{-1}\!\sqrt{t}\bigr)=2\bigl\{\xi^2(t)+\eta^2(t)\bigr\}.
\]
With $\phi(x)$ defined in Theorem \ref{measure3} we obtain for any polynomial $p(x)$ of degree at most $m+n-2$
\begin{equation}\label{integrals}
    \int\limits_{\mathbb{R}}\frac{p(x)\operatorname{Im}\phi(x)}{\bigl|\sqrt{1-x^2}\,\xi(x)-(\phi(x)-x)\eta(x)\bigr|^2}\, dx=\int\limits_{-1}^1 \frac{p(t)}{\xi^2(t)+\eta^2(t)}\frac{dt}{\sqrt{1-t^2}}.
\end{equation}

Note that the limiting form $m,n\to\infty$ of \eqref{integrals} is
\begin{equation}\label{integrals2}
    \frac{1}{2}\int\limits_{\mathbb{R}}\frac{p(x)\operatorname{Im}\phi(x)}{\bigl|\cos(\sqrt{x})\cosh(\alpha \sqrt{x})-\phi(x)\sin(\sqrt{x})\sinh(\alpha \sqrt{x})\bigr|^2}\, dx=\int\limits_{\mathbb{R}}\frac{p(x)}{\cos(2\sqrt{x})+\cosh(2\alpha \sqrt{x})}\, dx.
\end{equation}
More general formulas similar to \eqref{integrals2}, their connection with the moment problem and in particular Nevanlinna parametrization are discussed in \cite{ismail}, \cite{kuznetsov2}.

\section{A finite analog of the integral in Ramanujan's question 353}\label{rq}
\begin{Th} Let $n$ be a positive even integer and $k$ a positive odd integer. Then
    \begin{equation}\label{353m}
    \int\limits_{0}^{1}\frac{\sin (kn \sin^{-1}{t})}{\cos (n \sin^{-1}{t})+\cosh (n \sinh^{-1}{t})}\,\frac{dt}{t}=\frac{\pi}{4}.
\end{equation}
\end{Th}
\begin{proof} Let $n=2\nu$, $k=2\mu+1$, where $\nu$ is a positive integer, and $\mu$ is a nonnegative integer. Similar to Section 3 of \cite{nicholson}, or by other means, one can derive the partial fractions expansion
\[
    \frac{1}{\cos (2\nu \sin^{-1}{t})+\cosh (2\nu \sinh^{-1}{t})}\,\frac{\sin (2\nu \sin^{-1}{t})}{t\sqrt{1-t^2}}=\frac{1}{\nu}\sum_{j=1}^{\nu}\frac{i-\cos\frac{\pi(2j-1)}{2\nu}}{2t^2\cos\frac{\pi(2j-1)}{2\nu}+i\sin^2\frac{\pi(2j-1)}{2\nu}}.
\]
Further calculations assume that $\nu$ is even. When $\nu$ is odd, calculations are similar, except that one has to take special care of the term with $j=(\nu+1)/2$. Thus, define 
\[
q_j=\frac{1-\sin\frac{\pi(2j-1)}{2\nu}}{\cos\frac{\pi(2j-1)}{2\nu}}\,e^{-i\frac{\pi(2j-1)}{2\nu}},\qquad j=1,2,\ldots, \nu.
\]
Obviously,
\[
|q_j|<1, \qquad j=1,2,\ldots, \nu.
\]
We are going to make the change of variables
\[
t=\sin(\varphi/2),\quad \varphi\in(0,\pi),
\]
in the integral \eqref{353m}. Thus $2t^2=1-\cos\varphi$, and $4\sqrt{1-t^2}\,dt=(1+\cos\varphi)\,d\varphi$. By simple algebra
\[
\frac{1+\cos\varphi}{(1-\cos\varphi)\cos\frac{\pi(2j-1)}{2\nu}+i\sin^2\frac{\pi(2j-1)}{2\nu}}=\frac{-2}{(1-q_j)\cos\frac{\pi(2j-1)}{2\nu}}\left(1+(1+q_j)\frac{1-q_j\cos\varphi}{1-2q_j\cos\varphi+q_j^2}\right).
\]
According to well known formulas
\[
\frac{1-q_j\cos\varphi}{1-2q_j\cos\varphi+q_j^2}=\sum_{r=0}^\infty q_j^r\cos(r\varphi),
\]
\[
\frac{\sin(k\nu\varphi)}{\sin(\nu\varphi)}=1+2\sum_{l=1}^{\mu}\cos(2\nu l\varphi).
\]
Thus, the integral \eqref{353m} becomes
\[
I=\frac{1}{2\nu}\sum_{j=1}^{\nu}\frac{\cos\frac{\pi(2j-1)}{2\nu}-i}{(1-q_j)\cos\frac{\pi(2j-1)}{2\nu}}\int\limits_0^{\pi}\left(1+(1+q_j)\sum_{r=0}^\infty q_j^r\cos(r\varphi)\right)\left(1+2\sum_{l=1}^{\mu}\cos(2\nu l\varphi)\right)\,d\varphi.
\]
The integrals are easily calculated using orthogonality of cosines on $(0,\pi)$:
\[
I=\frac{\pi}{2\nu}\sum_{j=1}^{\nu}f(j),\qquad f(j)=\frac{\cos\frac{\pi(2j-1)}{2\nu}-i}{(1-q_j)\cos\frac{\pi(2j-1)}{2\nu}}\left(2+q_j+(1+q_j)\sum_{l=1}^\mu q_j^{2\nu l}\right).
\]
Trivial algebra (under the transformation $j\to \nu+1-j$ the expression $\cos\frac{\pi(2j-1)}{2\nu}$ changes sign, while $\sin\frac{\pi(2j-1)}{2\nu}$ and $q_j^{2\nu}$ do not change) shows that
\[
f(j)+f(\nu+1-j)=1, \qquad j=1,2,\ldots, \nu.
\]
Hence $\sum_{j=1}^{\nu}f(j)=\nu/2$, and $I=\pi/4$.
\end{proof}

One can obtain a finite analog of Theorem 4.2 from \cite{berndt} multiplying the integrand in \eqref{353m} by $t^{4b}$, $b\in \mathbb{N}$.

\section{Finite analogs of integrals related to theta series}\label{theta1}
Consider the finite trigonometric sum
\begin{equation}\label{trig_sum}
    S(n,m)=\sum _{j=0}^{n/2} (-1)^{j} \sin \frac{\pi  (2 j+1)}{2 n}\,\Big\{\!\cos\frac{\pi  (2 j+1)}{2 n}\Big\}^{m-1}.
\end{equation}

\begin{Th}\label{tt} Let $m>1$ and $n$ be positive odd integers. Then
    \begin{equation}\label{theta}
    \int\limits_{0}^{1}\frac{\sin \bigl(n \sin^{-1}\sqrt{t}\bigr)\sinh \bigl(n \sinh^{-1}\sqrt{t}\bigr)}{\cos \bigl(2n \sin^{-1}\sqrt{t}\bigr)+\cosh \bigl(2n \sinh^{-1}\sqrt{t}\bigr)}\,\frac{\sin\bigl(m\sin^{-1}t\bigr)}{\sqrt{1-t^2}}\,dt=\frac{\pi}{4n}\,S(n,m).
\end{equation}
\end{Th}
\begin{proof} We make the change of variables $t=\sin\theta$ in the integral. From
\[
 \sin^{-1}\sqrt{\sin\theta}+i\sinh^{-1}\sqrt{\sin\theta}=\cos^{-1}\bigl(e^{-i\theta}\bigr),\qquad 0\le\theta\le\frac{\pi}{2}
\]
(Section \ref{function}), and the well known partial fractions expansion formula for the reciprocal of the Chebyshev polynomial of the first kind
\[
\frac{1}{T_k(z)}=\frac{1}{k}\sum _{j=0}^{k-1} (-1)^{j} \frac{\sin\frac{\pi  (2 j+1)}{2 k}}{z-\cos\frac{\pi  (2 j+1)}{2 k}},
\]
it follows that
\begin{equation}\label{pf0}
    \frac{1}{\cos \bigl(k \sin^{-1}\sqrt{\sin\theta}+ik \sinh^{-1}\sqrt{\sin\theta}\bigr)}=\frac{1}{k}\sum _{j=0}^{k-1} (-1)^{j} \frac{\sin\frac{\pi  (2 j+1)}{2 k}}{e^{-i\theta}-\cos\frac{\pi  (2 j+1)}{2 k}},\qquad k\in\mathbb{N}.
\end{equation}
For odd $n$ one obtains from this
\addtocounter{equation}{1}
\begin{equation}\label{tsgf3}
    \frac{1}{\cos \bigl(n \sin^{-1}\sqrt{\sin\theta}+in \sinh^{-1}\sqrt{\sin\theta}\bigr)}=\frac{1}{n}\sum_{r=1}^\infty S(n,2r+1)\,e^{i(2r+1)\theta},\tag{44}
\end{equation}
with $S(n,2r+1)$ defined in \eqref{trig_sum}. Taking the imaginary part of \eqref{tsgf3} results in
\begin{equation}\label{tsgf1}
    \frac{\sin \bigl(n \sin^{-1}\sqrt{\sin\theta}\bigr)\sinh \bigl(n \sinh^{-1}\sqrt{\sin\theta}\bigr)}{\cos \bigl(2n \sin^{-1}\sqrt{\sin\theta}\bigr)+\cosh \bigl(2n \sinh^{-1}\sqrt{\sin\theta}\bigr)}=\frac{1}{n}\sum_{r=1}^\infty S(n,2r+1)\sin(2r+1)\theta.\tag{44a}
\end{equation}
Thus the integral becomes
\[
\frac{1}{n}\sum_{r=1}^\infty S(n,2r+1)\int\limits_0^{\pi/2}\sin\{(2r+1)\theta\}\sin(m\theta)\,d\theta.
\]
Now the result follows by orthogonality of the system of functions $\big\{\!\sin{(2k+1)\theta}\big\}_{k=0}^\infty$ on $[0,\pi/2]$ (system of eigenfunctions of the Sturm-Liouville problem $f''+\lambda f=0,~f(0)=f'(\pi/2)=0$).
\end{proof}
The companion formula to \eqref{tsgf1}, which follows from \eqref{tsgf3} by taking the real part, is
\begin{equation}\label{tsgf2}
    \frac{\cos \bigl(n \sin^{-1}\sqrt{\sin\theta}\bigr)\cosh \bigl(n \sinh^{-1}\sqrt{\sin\theta}\bigr)}{\cos \bigl(2n \sin^{-1}\sqrt{\sin\theta}\bigr)+\cosh \bigl(2n \sinh^{-1}\sqrt{\sin\theta}\bigr)}=\frac{1}{n}\sum_{r=1}^\infty S(n,2r+1)\cos(2r+1)\theta,\tag{44b}
\end{equation}
with the same $S(n,2r+1)$ as in \eqref{trig_sum}, and for odd $n$. It leads to integration formula analogous to \eqref{theta}. Moreover, if we replace $m$ by $mn$ in \eqref{theta}, multiply both sides by $n^2$, make change of variables in the integral $t\to t/n^2$ and let $n,m\to\infty$ such that $m/n\to a$ in the resulting formula, we get
\[
\int\limits_0^\infty \frac{\sin \bigl(\sqrt{x}\bigr)\sinh \bigl(\sqrt{x}\bigr)}{\cos \bigl(2 \sqrt{x}\bigr)+\cosh \bigl(2\sqrt{x}\bigr)}\,\sin(a x)\,dx=\frac{\pi^2}{8}\sum_{j=0}^\infty(-1)^{j}(2j+1)\,e^{-{\pi^2 (2j+1)^2a}/{8}}
\] 
(for justification of the limiting process for a similar series see \cite{polya}). This means that \eqref{theta} and its companion formula are finite analogs of two integrals due to Glaisher, equations (21) and (22) in \cite{glaisher}.

There is extensive literature on finite trigonometric sums. Sums similar to \eqref{trig_sum} are studied for example in \cite{fonseca} and \cite{cadavid} using various techniques. These papers also discuss different contexts such sums arise in, and also contain survey of the earlier literature on the topic. In particular, finite trigonometric sums have been studied using generating functions. For example, generating functions for finite sums (of reciprocals) of trigonometric functions were calculated in \cite{wang} (this paper generalizes some generating functions from the earlier literature). Equations \eqref{tsgf3}, \eqref{tsgf1} and \eqref{tsgf2} are generating functions for the finite trigonometric sum \eqref{trig_sum}. In fact, Theorem \ref{tt} is equivalent to the generating function \eqref{tsgf1}. The generating functions in \cite{wang} contain an extra free parameter, moreover they are simpler in form in the sense that they contain just one Chebyshev polynomial in the denominator. Equations \eqref{tsgf1} and \eqref{tsgf2} contain two Chebyshev polynomials in the denominator, of the same order but of different arguments. Functions that contain two Chebyshev polynomials of different orders but of the same argument arise as generating functions for probablities of random walks with boundaries, (\cite{giuggioli}, equations (2),(3),(12)).

There is a one-parameter generalization of \eqref{pf0}:
\begin{Th} Let $c\in[0,1)$, $\theta\in[0,\pi]$, and $n\in\mathbb{N}$. Then
\begin{equation}\label{pf}
    \bigg\{\!\cos \bigg(\!n \sin^{-1}\!\sqrt{\frac{\sin\theta+c}{1+c}}+in \sinh^{-1}\!\sqrt{\frac{\sin\theta+c}{1-c}}\bigg)\!\bigg\}^{-1}=\frac{1}{n}\sum _{j=0}^{n-1} \frac{(-1)^{j} \sqrt{1-c^2}\,\sin\frac{\pi  (2 j+1)}{2 n}}{e^{-i\theta}-\sqrt{1-c^2}\cos\frac{\pi  (2 j+1)}{2n}-ic}.
\end{equation}
\end{Th}
A generalization of Theorem \ref{tt} can be obtained from \eqref{pf}.

\section{Weights with the difference \texorpdfstring{$\cosh-\cos$}{cosh-cos} in the denominator}\label{appendix}

Consider the positive weight function on $\mathbb{R}$
\[
\frac{x}{\cosh \bigl(2\sqrt{x/a}\bigr)-\cos \bigl(2 \sqrt{x}\bigr)},\qquad a>0,
\]
which has a removable singularity at $x=0$. The symmetric $a=1$ case of the integrals with such functions was studied in \cite{berndt}, and the non-symmetric case in \cite{bradshaw}. The weight functions on $[-a,1]$
\[
\frac{t}{\cosh \bigl(2 m \sinh^{-1}\!\sqrt{t/a}\bigr)-\cos \bigl(2 n \sin^{-1}\!\sqrt{t}\bigr)}\frac{1}{\sqrt{(1-t)(a+t)}}
\]
approximate the above functions in the limit $n,m\to\infty$. Below we briefly discuss how do the formulas in the main text change for such weight functions.

Define the function
\[
\rho_a(t)=\frac{\cosh \bigl(2 m \sinh^{-1}\!\sqrt{t/a}\bigr)-\cos \bigl(2 n \sin^{-1}\!\sqrt{t}\bigr)}{t}
=\biggl|\sqrt{\frac{2}{t}}\,\sin \bigl( n \sin^{-1}\sqrt{t}-im \sinh^{-1}\sqrt{t/a}\bigr)\biggr|^2.
\]
Using the substitution
\[
t=\frac{1}{2}\bigl\{1-a+(1+a)\cos\theta\bigr\},\quad \theta\in[0,\pi],
\]
one can prove that
\[
\sqrt{\frac{2}{t}}\,\sin \bigl( n \sin^{-1}\sqrt{t}-im \sinh^{-1}\sqrt{t/a}\bigr)=i^{n-1}\,e^{-i (n+m-1)\theta/2}\,h_a\bigl(e^{i\theta}\bigr),
\]
\begin{align*}
    \nonumber h_a(z)&=\sqrt{\frac{2}{a+1}}\left(\frac{a+1}{2\sqrt{a}}\right)^{m}\{\sigma(z)\}^{-1}\\&\times \left\{\bigl[\sigma(z)+1+cz\bigr]^{\frac{m+n}{2}} \bigl[\sigma(z)+c+z\bigr]^{\frac{m-n}{2}}
    -(-1)^n\bigl[\sigma(z)-1-cz\bigr]^{\frac{m+n}{2}} \bigl[\sigma(z)-c-z\bigr]^{\frac{m-n}{2}}\right\},
\end{align*}
where
\[
\sigma(z)=\sqrt{z^2+2cz+1},\qquad c=\frac{1-a}{1+a}.
\]
$h_a(z)$ is a polynomial in $z$. One can show that $h_a(z)$ satisfies the conditions of Lemma \ref{lemma} (see also the note at the end of Section \ref{pth1}). Thus the normalized representation required by Theorem \ref{sb} in this case is
\[
{\rho}(\cos\theta)=\rho_a(t)=|h_a(e^{i\theta})|^2.
\]

Using \eqref{pk} we find the following orthogonal polynomials with known roots for the weight function under consideration: When $n$ is odd and $m$ is even
\[
p_{\frac{m+n-1}{2}}(t)=\frac{2}{\sqrt{\pi t}}\,\sin \bigl(n \sin^{-1}\!\sqrt{t}\bigr)\cosh \bigl(m \sinh^{-1}\!\sqrt{t/a}\bigr);
\]
when $n$ is even and $m$ is odd
\[
p_{\frac{m+n-1}{2}}(t)=\frac{2}{\sqrt{\pi t}}\,\cos \bigl(n \sin^{-1}\!\sqrt{t}\bigr)\sinh \bigl(m \sinh^{-1}\!\sqrt{t/a}\bigr).
\]

Now we can find explicit Gauss quadrature formulas. If $n$ is odd and $m$ is even, and $p(t)$ is a polynomial of degree at most $m+n-1$ such that $p(0)=0$, then
\begin{align*}
    \nonumber\int\limits_{-a}^{1}&\frac{p(t)}{\cosh \bigl(2 m \sinh^{-1}\!\sqrt{t/a}\bigr)-\cos \bigl(2 n \sin^{-1}\!\sqrt{t}\bigr)}\frac{dt}{\sqrt{(1-t)(a+t)}}\\
    &=\frac{2 \pi }{n}\sum_{i=1}^{n/2}\frac{\tanh\frac{\alpha_{2i}}{2n}}{\sinh\bigl(\frac{m}{n}\alpha_{2i}\bigr)}\,p\bigl(\sin^2\tfrac{\pi i}{n}\bigr)-\frac{2 \pi}{m}\sum_{j=1}^{m/2}\frac{\tanh\frac{\beta_{2j-1}}{2m}}{\sinh\bigl(\frac{n}{m}\beta_{2j-1}\bigr)}\,p\bigl(-a\sin^2\tfrac{\pi (2j-1)}{2m}\bigr),
\end{align*}
where $\alpha_z$ and $\beta_z$ are defined in \eqref{ab}. Quadrature formula for even $n$ and odd $m$ is obtained from the above formula by the change of variables $t\to -t/a$, and subsequent redefinitions of the parameters $a\to 1/a$, $m\to n$, $n\to m$ and the polynomial $p(t)$. The integral has the same limiting form as $n,m\to\infty$ irrespective of the parity of the integers $n$ and $m$: For any polynomial $p(x)$ and $\alpha>0$
\begin{align*}
    \frac{1}{4\pi^4}\int\limits_{\mathbb{R}} \frac{xp(x)}{\!\cosh \bigl(\alpha\sqrt{x}\bigr)-\cos \bigl(\sqrt{x}\bigr)}\,dx&=\sum_{\substack{j\ge 1\\\mathrm{even}} }\frac{j^3}{\sinh(\pi\alpha j)}\,p\bigl(\pi^2j^2\bigr)+\frac{1}{\alpha^4}\sum_{\substack{j\ge 1\\\mathrm{odd}} }\frac{j^3}{\sinh\frac{\pi j}{\alpha}}\,p\Bigl(-\tfrac{\pi^2j^2}{\alpha^{2}}\Bigr)\\
    &=\sum_{\substack{j\ge 1\\\mathrm{odd}} }\frac{j^3}{\sinh(\pi\alpha j)}\,p\bigl(\pi^2j^2\bigr)+\frac{1}{\alpha^4}\sum_{\substack{j\ge 1\\\mathrm{even}} }\frac{j^3}{\sinh\frac{\pi j}{\alpha}}\,p\Bigl(-\tfrac{\pi^2j^2}{\alpha^{2}}\Bigr).
\end{align*} 
Detailed study of the corresponding indeterminate moment problem is given in \cite{ivy} (in particular, Theorem 4.4.4), where parametrization in terms of elliptic integrals is used: $\alpha=K'/K$.

A particular form of the quadrature formula is: Let $n$ and $m$ be positive integers, $u$ be an odd integer, and let $\alpha_z$ and $\beta_z$ be defined as in \eqref{ab}, then
\begin{align*}
    \int\limits_{-a}^1&\frac{\sqrt{t}\,\sin \bigl(u\sin ^{-1}\sqrt{t}\bigr)}{\cosh \bigl(2 m \sinh ^{-1}\sqrt{t/a}\bigr)-\cos \bigl(2 n \sin ^{-1}\sqrt{t}\bigr)}\frac{dt}{\sqrt{(1-t)(a+t)}}\\&{\phantom{......................}}=\frac{\pi}{2n\sqrt{a}}\sum_{j=1}^{2n-1}(-1)^{j}\,\frac{\sin^2\frac{\pi j}{2n}}{\cosh\frac{\alpha_j}{2n}}\left\{\coth\frac{m\alpha_j}{2 n}\right\} ^{(-1)^j}\!\!\!\!\cdot\sin\frac{\pi j u}{2n},\qquad |u|<2n-2,\\
    &{\phantom{......................}}=\frac{\pi a}{2n}\sum_{j=1}^{2m-1}(-1)^{j}\,\frac{\sin^2\frac{\pi j}{2m}}{\cosh\frac{\beta_j}{2m}}\left\{\coth\frac{n\beta_j}{2 m}\right\} ^{(-1)^j}\!\!\!\!\cdot\sinh\frac{\beta_j u}{2m},\qquad |u|<2m-2.
\end{align*}
This is a finite analog of the generating function derived in \cite{bradshaw}.

Now consider the weight function
\[
\sqrt{(1-t)(a+t)}\,\{\widetilde{\rho}_a(t)\}^{-1},
\]
where
\[
\widetilde{\rho}_a(t)=\frac{1}{t^2}\bigl\{\cosh \bigl(2 m \sinh^{-1}\!\sqrt{t/a}\bigr)-\cos \bigl(2 n \sin^{-1}\!\sqrt{t}\bigr)\bigr\}\bigl\{\cosh \bigl(2 m' \sinh^{-1}\!\sqrt{t/a}\bigr)-\cos \bigl(2 n \sin^{-1}\!\sqrt{t}\bigr)\bigr\}.
\]
When $m+m'$ is even, we find the following orthogonal polynomial for this weight function
\[
p_{\frac{m+m'}{2}+n-2}(t)=\sqrt{\frac{2}{\pi}}\,\frac{\sin \bigl(2n \sin^{-1}\!\sqrt{t}\bigr)\sinh \bigl((m+m') \sinh^{-1}\!\sqrt{t/a}\bigr)}{t\sqrt{(1-t)(a+t)}}.
\]
The limiting form of the corresponding quadrature formula is
\begin{align*}
    &\frac{1}{2\pi^6}\int\limits_{\mathbb{R}} \frac{x^2p(x)}{\big\{\!\cosh \bigl(\alpha\sqrt{x}\bigr)-\cos \bigl(\sqrt{x}\bigr)\big\}\big\{\!\cosh \bigl(\beta\sqrt{x}\bigr)-\cos \bigl(\sqrt{x}\bigr)\big\}}\, dx\\
    &=\sum_{j\ge 1}\frac{ j^5}{\sinh\frac{\pi(\alpha+\beta)j}{2}}\,\frac{p\bigl(\pi^2j^2\bigr)}{\cosh\frac{\pi(\alpha+\beta)j}{2}-(-1)^j\cosh\frac{\pi(\alpha-\beta)j}{2}}+\frac{64}{(\alpha+\beta)^6}\sum_{j\ge 1}\frac{j^5}{\sinh\frac{2\pi j}{\alpha+\beta}}\,\frac{p\left(-\frac{4\pi^2j^2}{(\alpha+\beta)^2}\right)}{\cosh\frac{2\pi j}{\alpha+\beta}-\cos\frac{2\pi\alpha j}{\alpha+\beta}}.
\end{align*} 

One could also consider mixed weight functions with
\[
\widetilde{\rho}_a(t)=\frac{1}{t}\bigl\{\cosh \bigl(2 m \sinh^{-1}\!\sqrt{t/a}\bigr)+\cos \bigl(2 n \sin^{-1}\!\sqrt{t}\bigr)\bigr\}\bigl\{\cosh \bigl(2 m' \sinh^{-1}\!\sqrt{t/a}\bigr)-\cos \bigl(2 n \sin^{-1}\!\sqrt{t}\bigr)\bigr\}.
\]
When $m+m'$ is even, the corresponding orthogonal polynomial is
\[
\frac{\sin \bigl(2n \sin^{-1}\!\sqrt{t}\bigr)\cosh \bigl((m+m') \sinh^{-1}\!\sqrt{t/a}\bigr)}{\sqrt{t(1-t)}}.
\]
The limiting form of the quadrature formula is
\begin{align*}
    &\frac{1}{2\pi^4}\int\limits_{\mathbb{R}} \frac{xp(x)}{\big\{\!\cosh \bigl(\alpha\sqrt{x}\bigr)+\cos \bigl(\sqrt{x}\bigr)\big\}\big\{\!\cosh \bigl(\beta\sqrt{x}\bigr)-\cos \bigl(\sqrt{x}\bigr)\big\}}\, dx\\
    &=\sum_{j\ge1}\frac{ j^3}{\cosh\frac{\pi(\alpha+\beta)j}{2}}\,\frac{p\bigl(\pi^2j^2\bigr)}{\sinh\frac{\pi(\alpha+\beta)j}{2}-(-1)^j\sinh\frac{\pi(\alpha-\beta)j}{2}}+\frac{2}{(\alpha+\beta)^4}\sum_{\substack{j\ge 1\\ \mathrm{odd}}}\frac{j^3}{\sinh\frac{\pi j}{\alpha+\beta}}\,\frac{p\left(-\frac{\pi^2j^2}{(\alpha+\beta)^2}\right)}{\cosh\frac{\pi j}{\alpha+\beta}+\cos\frac{\pi\alpha j}{\alpha+\beta}}.
\end{align*} 

Similar to \eqref{pf0}, one can deduce from the partial fractions expansion formula for the reciprocal of the Chebyshev polynomial of the second kind
\[
\frac{1}{(1-z^2)U_{k-1}(z)}=\frac{1}{2k}\sum _{j=1}^{2k}  \frac{(-1)^{j-1}}{z-\cos\frac{\pi  j}{k}},
\]
that
\[
\frac{e^{i\theta/2}}{\sqrt{2i\sin\theta}}\cdot\frac{1}{\sin \bigl(k \sin^{-1}\sqrt{\sin\theta}+ik \sinh^{-1}\sqrt{\sin\theta}\bigr)}=\frac{1}{2k}\sum _{j=1}^{2k} \frac{(-1)^{j-1} }{e^{-i\theta}-\cos\frac{\pi j}{k}},\qquad k\in\mathbb{N}.
\]

\section*{Acknowledgements}

We would like to thank Teruo \cite{teruo} and Po1ynomial \cite{po1ynomial} from math.stackexchange.com website for their stimulating and insightful answers to our question \#3595770.

\end{document}